\documentclass[reqno]{amsart}

\pdfoutput=1

\usepackage{amsfonts,amsmath,amsthm,amsxtra,amssymb,mathrsfs,dsfont}

\usepackage{tikz-cd}

\usepackage[square,numbers]{natbib}

\usepackage{graphicx}  

\usepackage{subfigure}

\usepackage{fancybox}

\usepackage[colorlinks = true, citecolor = blue, urlcolor = blue]{hyperref}

\newtheorem{theorem}{Theorem}[section]
\newtheorem{lemma}[theorem]{Lemma}
\newtheorem{proposition}[theorem]{Proposition}

\theoremstyle{definition}
\newtheorem{definition}[theorem]{Definition}
\newtheorem{remark}[theorem]{Remark}

\def\gor#1{\widetilde{#1}}

\def\bar#1{\overline{#1}}

\def\h(#1,#2){\mbox{Hom}\left(#1,#2\right)}

\def\t(#1,#2){\mbox{Tor}\left(#1,#2\right)}
\def\e(#1,#2){\mbox{Ext}\left(#1,#2\right)}
\def\mod{\;\mathsf{mod}\;}

\def\C{\mathbb{C}}

\def\M{\mathbb{M}}
\def\N{\mathbb{N}}

\def\R{\mathbb{R}}

\def\X{\mathbb{X}}

\def\Z{\mathbb{Z}}

\def\dd{\mathbf{d}}

\def\dgm{\mathsf{dgm}}

\def\<{\langle}
\def\>{\rangle}

\DeclareMathOperator*{\argmax}{\arg\!\max\;}

\title[Sparse Circular Coordinates]{Sparse Circular Coordinates via Principal $\mathbb{Z}$-Bundles}

\author[Jose Perea]{Jose A. Perea}
\thanks{This work was partially supported by the NSF under grant DMS-1622301 and  DARPA under grant HR0011-16-2-003}
\address{
\shortstack[l]{
Department of Computational Mathematics, Science \& Engineering \\
Department of Mathematics, \\
Michigan State University \\
East Lansing, MI, USA.}}
\email{joperea@msu.edu}



\begin{document}

\bibliographystyle{abbrvnat}
\begin{abstract}
We present in this paper an application of the theory of principal bundles
to the problem of nonlinear dimensionality reduction in data analysis.
More explicitly, we derive, from a  1-dimensional persistent cohomology computation,
explicit formulas for
circle-valued functions on data
with nontrivial underlying topology.
We  show that  the language of principal bundles leads to coordinates defined
on an open neighborhood of the data, but computed using only a smaller subset of landmarks.
It is in this sense that the coordinates are sparse.
Several data examples are presented, as well as  theoretical results underlying  the construction.
\end{abstract}

\maketitle
\section{Introduction}
The curse of dimensionality
refers to a host of phenomena inherent to the increase in the number of features describing the elements of a data sets.
For instance, in statistical learning, the number of training data points needs to grown roughly exponentially in the
number of features, in order for learning algorithms to generalize correctly in the absence of other priors.
A deeper manifestation of the curse of dimensionality is the deterioration of the concept of ``nearest neighbors'' in high-dimensional Euclidean space;
for as the dimension increases, the distance between any two points is roughly the same \cite{radovanovic2010hubs}.
One of the most popular priors in data science
is the ``low intrinsic dimensionality'' hypothesis.
It contends that while the apparent number of features  describing each data  point
 (e.g., the number of pixels in an image) might be large, the effective number of degrees of freedom (i.e., the intrinsic dimensionality) is often much lower.
Indeed,  images generated at random will hardly depict a cat or a natural scene.

Many dimensionality reduction schemes have been proposed in
the literature
  to leverage the ``low intrinsic dimensionality''  hypothesis,
each   making explicit or implicit use of  likely characteristics of the data.
For instance, Principal Component Analysis \cite{jolliffe2002principal} and other linear dimensionality reduction methods,
rely on the existence of a low-dimensional linear representation accounting for most of the variability in the data.
Methods such as Locally Linear Embeddings \cite{roweis2000nonlinear} and
Laplacian EingenMaps \cite{belkin2003laplacian}, on the other hand, presuppose  the existence of a manifold-like object parametrizing the underlying data space.
Other algorithms, like  Multidimensional Scaling \cite{kruskal1964multidimensional} and Isomap \cite{tenenbaum2000global},
attempt to preserve distances between data points while providing low-dimensional reconstructions.

Recently, several new methods for nonlinear dimensionality reduction have emerged from the field of computational topology \cite{de2011persistent, singh2007topological, perea2018multiscale}.
The idea being that if the underlying space from which the data has been sampled has a particular shape,
then this information can be used to generate appropriate low-dimensional representations.
The circular coordinates of de Silva,  Morozov, and Vejdemo-Johansson \cite{de2011persistent}  pioneered the use of persistent
cohomology as a way to measure the shape of a data set, and then produce circle-valued coordinates reflecting the underlying nontrivial topology.
Their algorithm goes as follows.
Given  a finite metric space $(X,\dd)$ --- the data --- and a scale $\alpha > 0$ so that  the Rips complex
\[
R_{\alpha}(X)  :=\{ \sigma \subset X : \sigma \neq \emptyset\;\; \mbox{ and }\;\; \mathsf{diam}(\sigma) < \alpha\}
\]
has a nontrivial  integer cohomology class $[\eta] \in H^1(R_\alpha(X); \Z)$
--- this is determined from the persistent cohomology of the Rips filtration $\mathcal{R}(X) = \{R_\epsilon(X)\}_{\epsilon \geq 0}$ ---
a linear least squares optimization   (of size the number of vertices by the number of edges of $R_\alpha(X)$) is solved,
in order to construct a function $f_\eta : X \longrightarrow S^1 \subset \C$ which, roughly,
 puts one of the generators of $H^1(S^1;\Z)\cong \Z$ in correspondence with $[\eta] \in H^1(R_\alpha(X);\Z)$.

\subsection{Our Contribution} Two drawbacks of the perspective presented in \cite{de2011persistent} are:
(1) the method requires a persistent cohomology
calculation, as well as a least squares optimization,  on the Rips filtration of the entire data set $X$. This is computationally expensive and may limit  applicability to small-to-medium-sized data.
(2) once the function $f_\eta$ has been computed, it is only defined on the data points from $X$ used for its construction.
Here we show  that these drawbacks can be addressed effectively with ideas from
principal $\Z$-bundles. In particular, we show  that it is possible to construct circular coordinates on $X$ from the Rips filtration
on a subset of landmarks $L\subset X$, Proposition \ref{prop:unnecesary},
with similar classifying properties as in \cite{de2011persistent}, Theorem \ref{thm:ClassifyingForumlaHarmonic},
and that said coordinates will be defined on an open neighborhood of $L$ containing $X$.
We call these functions ``sparse circular coordinates''.

\subsection{The Sparse Circular Coordinates Algorithm}\label{sec:SCCalgorithm}
Let us describe next the steps needed to construct said coordinates.
The rest of the paper is devoted to the theory behind these choices:
\begin{enumerate}
\item Let $(X,\dd)$ be the input data set; i.e. a finite metric space.
    Select a set of landmarks $L = \{\ell_1\ldots, \ell_N\} \subset X$, e.g. at random or  via \verb"maxmin" sampling,
    and let \[r_L :=  \max\limits_{x\in X}\, \min\limits_{\ell \in L}\,\dd(x,\ell) \] be the radius of coverage.
    In particular, $r_L$ is the Hausdorff distance between $L$ and $X$.
\item Choose a prime $q>2$ at random and compute the 1-dimensional persistent cohomology $PH^1(\mathcal{R}(L) ; \Z/q)$
with coefficients in $\Z/q$, for the Rips filtration on the landmark set $L$.
    Let $\dgm(L)$ be the resulting persistence diagram.
\item If there exists $(a,b)\in \dgm(L)$
    so that $ \max\{a, r_L\} < \frac{b}{2}$, then let
    \[\alpha = t\cdot\max\{a, r_L\} + (1-t) \frac{b}{2}
    \;\;\; \;\; , \;\;\mbox{ for some }\;\;0< t<1\]
    Let $\eta' \in Z^1(R_{2\alpha}(L) ; \Z/q)$ be a cocyle representative for the persistent cohomology class
corresponding to $(a,b)\in \dgm(L)$.
    If $t$ is closer to 1, then the circular coordinates are defined on a larger domain; however,  this  makes  step (5) below more computationally intensive.
\item Lift $\eta' : C_1(R_{2\alpha}(L);\Z)\longrightarrow \Z/q = \{0,\ldots, q-1\}$ to an integer cocycle
    $\eta \in Z^1(R_{2\alpha}(L);\Z)$. That is, one for which $\eta' - (\eta \mod q)$ is a coboundary in $C^1(R_{2\alpha}(L);\Z/q)$.
    An explicit  choice (that works in practice for a prime $q$ chosen at random) is the integer  cochain:
\[
\eta(\sigma) =
\left\{
  \begin{array}{lcr}
    \eta'(\sigma) & \hbox{if } & \eta'(\sigma) \leq \frac{q-1}{2}  \\[.25cm]
    \eta'(\sigma) - q & \hbox{ if }& \eta'(\sigma) > \frac{q-1}{2}
  \end{array}
\right.
\]
\item Choose positive weights for the vertices and edges of $R_{2\alpha}(L)$ --- e.g. all equal to one ---
   and let $d_{2\alpha}^+ :C^1(R_{2\alpha}(L);\R) \longrightarrow C^0(R_{2\alpha}(L);\R)$
 be the (weighted) Moore-Penrose pseudoinverse (solving weighted linear least squares problems) for the coboundary map
    \[
d_{2\alpha} : C^0(R_{2\alpha}(L);\R) \longrightarrow C^1(R_{2\alpha}(L);\R)\]
If $\iota: \Z \hookrightarrow \R$ is the inclusion, let
\[
\tau = - d_{2\alpha}^+\big(\iota\circ\eta\big)
\hspace{1cm}\mbox{ and } \hspace{1cm}
\theta =
(\iota\circ \eta) \, +\, d_{2\alpha}\left(\tau\right)
\]
\item Denote by $\tau_j \in \R$  the value of $\tau$ on the vertex $\ell_j\in L$,
and by $\theta_{jk}\in \R$  the value of $\theta$ on the oriented edge $[\ell_j, \ell_k] \in R_{2\alpha}(L)$.
If  we let
\[
\varphi_j(b) =
\frac{|\alpha - \dd(\ell_j, b)|_+}{ \sum\limits_{k =1}^N |\alpha - \dd(\ell_k,b)|_+}
\;\;\;\;\; \;\; \mbox{ where }\;\;\;\;\;\;\;
|r|_+ = \max\{r,0\}, \;\; r\in \R
\]
and $B_\alpha(\ell_k)$ denotes the open ball of radius $\alpha>0$ centered at $\ell_k\in L$,
then   the sparse circular coordinates are defined by the formula:
\begin{equation}\label{eq:Formula}
\boxed{
\begin{array}{cccl}
  h_{\theta,\tau}: &   \bigcup\limits_{k =1 }^N B_{\alpha}(\ell_k) & \longrightarrow & S^1 \subset \C \\
   & B_\alpha(\ell_j)\ni b & \mapsto &  \exp\left\{ 2\pi i\left ( \tau_j +    \sum\limits_{k=1}^N \varphi_k(b) \theta_{jk}  \right) \right\} \\[.3cm]
\end{array}
}
\end{equation}
If   $X $ is a subspace of an ambient metric space $\M$, then   the $B_\alpha(\ell_k)$'s
can be taken to be ambient metric balls.
This is why we call the circular coordinates \emph{sparse}; $h_{\theta,\tau}$ is computed using only $L$, but
its domain of definition is an open subset of $\M$ which, by construction, contains all of $X$.
\end{enumerate}

\subsection{Organization} We start in Section \ref{sec:Preliminaries} with a few preliminaries on principal bundles,
highlighting the main theorems needed in later parts of the paper.
We assume
familiarity with persistent cohomology (if not, see \cite{perea2018brief}),
as well as the definition of \v{C}ech cohomology with coefficients in a presheaf (see for instance \cite{miranda1995algebraic}).
Section \ref{sec:Formulas} is devoted to deriving the formulas --- e.g. (\ref{eq:Formula}) above --- which turn a 1-dimensional integer cohomology class
into a circle-valued function.
In Section \ref{sec:RealData} we describe how to make all this theory applicable to real data sets.
We present several experiments in Section \ref{sec:Experiments} with both real and synthetic data,
and end in Section \ref{sec:FinalRemarks} with a few final remarks.

\section{Preliminaries}\label{sec:Preliminaries}

\subsection{Principal Bundles}
We present here a terse introduction to principal bundles, with the main results we will need later in the paper.
In particular, the connection between principal bundles and \v{C}ech chomology, which allows for explicit computations,
and their classification theory via  homotopy classes of maps to classifying spaces.
The latter description will be used to generate our sparse circular coordinates.
We refer the interested reader to \cite{husemoller1966fibre} for a more thorough presentation.

Let $B$ be a connected and paracompact\footnote{ So that partitions of unity always exist} topological space with basepoint $b_0 \in B$.

\begin{definition}
A pair $(p, E)$, with $E$ a topological space and  $ p :E \longrightarrow B$ a continuous map,
is said to be a fiber bundle over $B$ with fiber $F = p^{-1}(b_0)$, if:
\begin{enumerate}
  \item $p$ is surjective
  \item Every point $b\in B$ has an open neighborhood $U \subset B$ and a homeomorphism $\rho_U : U\times F \longrightarrow p^{-1}(U)$,
  called a local trivialization around $b$,
    so that $p \circ \rho_U(b',e) = b'$ for every $(b',e) \in U \times F$.
\end{enumerate}
\end{definition}

The spaces $E$ and $B$ are called, respectively, the total and base space of the bundle, and $p$ is called the projection map.

\begin{definition}
Let $G$ be an abelian topological group whose operation we write additively.
A fiber bundle $p : E \longrightarrow B$ is said to be a principal $G$-bundle if:
\begin{enumerate}
\item The total space $E$ comes equipped with a fiberwise  free  right $G$-action. That is, a continuous map
\[\cdot : E\times G \longrightarrow E\]
satisfying the right-action axioms, with
  $p(e\cdot g) = p(e)$ for every pair $(e,g)\in E\times G$, and for which  $e\cdot g = e$ only if $g$ is the identity of $G$.
\item The induced fiberwise $G$-action $p^{-1}(b)\times G \longrightarrow p^{-1}(b)$ is transitive for every $b\in B$
in the base space.
\item The local trivializations  $\rho_U : U\times F \longrightarrow p^{-1}(U)$ can be chosen to be $G$-equivariant:
that is, so that  $\rho_U(b,e\cdot g) = \rho_U(b,e)\cdot g$,
for every $(b,e,g)\in U\times F \times G$.
\end{enumerate}
Two principal $G$-bundles  $p_j : E_j \longrightarrow B$, $j=1,2$, are said to be isomorphic,
if there exists a $G$-equivariant homeomorphism $\Phi: E_1 \longrightarrow E_2$ so that $p_2 \circ \Phi = p_1$.
This defines an equivalence relation on  principal $G$-bundles over $B$, and the set of isomorphism classes
is denoted $\mathsf{Prin}_G(B)$.
\end{definition}

Given a principal $G$-bundle $p: E \longrightarrow B$ and a system of ($G$-equivariant) local trivializations
$\left\{\rho_j : U_j \times F \longrightarrow p^{-1}(U_j)\right\}_{j\in J}$,
 we have that
\begin{equation}\label{eq:TransFunct}
\rho_k^{-1}\circ \rho_j : (U_j \cap U_k)\times F \longrightarrow (U_j \cap U_k)\times F
\end{equation}
is a $G$-equivariant homeomorphism whenever $U_j \cap U_k \neq \emptyset$.
Since the $G$-action on $E$ is fiberwise free and fiberwise transitive, then $\rho_k^{-1}\circ \rho_j$
induces a well-defined continuous map
\[
\rho_{jk} : U_j \cap U_k \longrightarrow G \;\; \;\; \; \;\;\;\;  j,k \in J
\]
defined by the equation
\begin{equation}\label{eq:TransFunctValues}
\rho_k^{-1}\circ \rho_j (b, e) = (b, e \cdot \rho_{jk}(b)) \;\;\; , \;\;\; \mbox{ for all }(b,e) \in (U_j\cap U_k)\times F.
\end{equation}


The  $\rho_{jk}$'s are called the transition functions for the $G$-bundle $(p, E)$
corresponding to the system of local trivializations $\{ \rho_j\}_{j\in J}$.
In fact, these transition functions define an element in the \v{C}ech cohomology  of $B$.
Indeed, for each open set $U\subset B$  let $\mathsf{Maps}(U,G)$ denote the set of continuous maps from $U$ to $G$.
Since $G$ is an abelian group, then so is $\mathsf{Maps}(U,G)$, and if $V \subset U$ is another open set,
then  precomposing with the inclusion $V\hookrightarrow U$ yields a restriction map
\[
\iota_{U,V} : \mathsf{Maps}(U,G) \longrightarrow \mathsf{Maps}(V,G)
\]
This defines a sheaf $\mathscr{C}_G$ of abelian groups over $B$, with $\mathscr{C}_G(U) := \mathsf{Maps}(U,G)$,
called the sheaf of $G$-valued continuous functions on $B$.
It follows that the transition functions (\ref{eq:TransFunct})
define an element   $\rho = \{\rho_{jk}\} \in \check{C}^1(\mathcal{U}; \mathscr{C}_G)$
in the \v{C}ech 1-cochains of the cover $\mathcal{U} =\{U_j\}_{j\in J}$ with coefficients in the sheaf $\mathscr{C}_G$.
Moreover,

\begin{proposition}
The transition functions $\rho_{jk}$ satisfy
the  cocycle condition
\begin{equation}\label{eq:CocycleCondition}
\rho_{j\ell}(b) = (\rho_{jk}+\rho_{k\ell})(b)  \;\;\;\mbox{ for all } \;\;\; b\in U_j \cap U_k \cap U_\ell
\end{equation}
In other words,
$\rho = \{\rho_{jk}\}\in \check{Z}^1(\mathcal{U}; \mathscr{C}_G)$
is a \v{C}ech cocycle.
\end{proposition}

If $\{\nu_r : V_r\times F \longrightarrow p^{-1}(V_r) \}_{r\in R}$ is another system of local trivializations
with induced \v{C}ech cocycle $\nu = \{\nu_{rs}\} \in \check{Z}^1(\mathcal{V};\mathscr{C}_G)$,
and    \[\mathcal{W} = \{U_j \cap V_r\}_{(j,r)\in J\times R}\] then one can check that the difference
$\rho - \nu $ is a coboundary in $ \check{C}^1(\mathcal{W}; \mathscr{C}_G)$.
Since $\mathcal{W}$ is a refinement for both $\mathcal{V}$ and $\mathcal{U}$,
it follows that the  $G$-bundle $p: E \longrightarrow B$ yields a well-defined element
$p_E \in \check{H}^1(B;\mathscr{C}_G)$. Moreover, after passing to isomorphism classes of $G$-bundles we get that
\begin{lemma}\label{lemma:GbundlesToCohomology}
The function
\[
\begin{array}{ccc}
\mathsf{Prin}_G(B) &\longrightarrow & \check{H}^1(B;\mathscr{C}_G) \\
\; [(p,E)]&\mapsto& p_E
\end{array}
\]
is well-defined and injective.
\end{lemma}

This is in fact a bijection.
To check surjectivity, fix an open cover $\mathcal{U}= \{U_j\}_{j\in J}$ for $B$,
and a \v{C}ech cocyle
\[
 \eta = \{\eta_{jk}\}  \in \check{Z}^1(\mathcal{U}; \mathscr{C}_G )
\]
Then one can construct a $G$-principal bundle over $B$ with total space
\begin{equation}\label{eq:BundleFromCocycle}
E_\eta = \left(\bigcup_{j\in J} U_j \times \{j\} \times G \right) \Big/ (b,j,g) \sim \big(b, k , g  + \eta_{jk}(b)\big)
\;\;\mbox{ , } \;\;
b\in U_j \cap U_k
\end{equation}
and projection
\[
p_\eta : E_\eta \longrightarrow B
\]
taking  the class of $(b, j, g) \in U_j \times \{j\} \times G$ in the quotient $E_\eta$, to the point $b\in B$.
Notice that if $\eta_j : U_j \times G \longrightarrow E_\eta$ sends
 $(b,g)$ to the class of $(b,j,g)$ in  $E_\eta$, then $\{\eta_j\}$ defines a system of local trivializations for $(p_\eta, E_\eta)$,
and that $\eta = \{\eta_{jk}\}$ is the associated system of transition functions.
Therefore,

\begin{theorem}\label{thm:IsoCohoPrinG}
The function
\[
\begin{array}{ccc}
\check{H}^1(B;\mathscr{C}_G) & \longrightarrow & \mathsf{Prin}_G(B) \\
\left[\eta\right] &\mapsto & [E_\eta]
\end{array}
\]
is a natural bijection.
\end{theorem}

\noindent In addition to this characterization of principal $G$-bundles over $B$  as \v{C}ech cohomology classes,
there is another  interpretation in terms of classifying maps.
We will combine these two views in order to produce coordinates for data in the next sections.

Indeed, to each topological group $G$ one can associate a space $EG$
that is both weakly contractible, i.e. all its homotopy groups are trivial, and which comes equipped with a free right $G$-action
\[
EG \times G \longrightarrow EG
\]
The quotient  $BG := EG/G$ is a topological space (endowed with the quotient topology),
called the classifying space of $G$,
and the quotient map
\[
\jmath:
EG \longrightarrow BG = EG/G
\]
defines a principal $G$-bundle over $BG$, called the universal bundle.
It is important to note that there are several constructions of $EG$, and thus of $BG$,
but they all have the same homotopy type.
One model for $EG$ is   the Milnor construction \cite{milnor1956construction}
\begin{equation}\label{eq:MilnorConstr}
\mathcal{E}G := G*G*G* \cdots
\end{equation}
with $G$ acting diagonally by right multiplication on each term of the infinite join.

The next Theorem explains the universality of $\jmath : EG \longrightarrow BG$.
Given a continuous map $f: B \longrightarrow BG$, the pullback
$f^*EG$ is   the principal $G$-bundle over $B$
with total space $\{(b,e) \in B\times EG : f(b) = \jmath(e)\}$,
and  projection map  $(b,e)\mapsto b$.
Moreover,
\begin{theorem}\label{thm:IsoHomoPrinG}
Let $[B,BG]$ denote the set of homotopy class of maps from $B$ to the classifying space $BG$.
Then, the function
\[
\begin{array}{ccl}
  [B, BG] & \longrightarrow &  \mathsf{Prin}_G(B)\\
 \; [ f] & \mapsto & [f^*EG]
\end{array}
\]
is a  bijection.
\end{theorem}
\begin{proof}
See \cite{husemoller1966fibre}, Chapter 4: Theorems 12.2 and 12.4.
\end{proof}
Theorem \ref{thm:IsoHomoPrinG} implies that given a principal $G$-bundle $p : E \longrightarrow B$,
there exists a continuous map $f: B \longrightarrow BG $ so
that $f^*EG $ is isomorphic to $(p,E)$, and that the choice of $f$ is unique up to homotopy.
Any such choice is called a classifying map for   $p: E \longrightarrow B$.

\section{From Integer Simplicial Cohomology to Circular Coordinates}\label{sec:Formulas}
For an arbitrary topological group $G$, the Milnor construction (\ref{eq:MilnorConstr})
produces an explicit universal $G$-bundle $\jmath: \mathcal{E}G\longrightarrow \mathcal{B}G$,
but the spaces $\mathcal{E}G$ and $\mathcal{B}G$ tend to be rather large.
Indeed, they are often infinite-dimensional CW-complexes.
For the case $G = \Z$ we have the more economical models $\mathcal{E}\Z\simeq \R$ and $\mathcal{B}\Z \simeq S^1 \subset \C$,
with $\Z$ acting on $\R$ by right translation: $\R \times \Z \ni (r,m)\mapsto r+m$,
and projection
\[
\begin{array}{rccl}
  p:  & \R & \longrightarrow &  S^1\\
    & r & \mapsto & \exp(2\pi i r)
\end{array}
\]

Since $\Z$ is discrete, then $\Z$-valued continuous functions on $B$ are in fact locally constant,
and hence $\mathscr{C}_\Z $ is exactly the sheaf of locally
constant functions with values in $\Z$, denoted $\underline{\Z}$.
Combining the definition of the \v{C}ech cohomology group $\check{H}^1(B;\underline{\Z})$   with
Theorems \ref{thm:IsoCohoPrinG} and \ref{thm:IsoHomoPrinG}, yields a bijection
\begin{equation}\label{eq:IsoCohoCirc}
\varprojlim_{\mathcal{U}} H^1(\mathcal{N}(\mathcal{U}); \Z) \cong \left[B, S^1\right]
\end{equation}
where the limit is taken over all locally finite covers $\mathcal{U}$ of $B$, ordered by refinement,
and the groups are the 1-dimensional simplicial cohomology with $\Z$ coefficients
 of the associated  nerve complexes
$\mathcal{N}(\mathcal{U})$.
The goal now  is to produce an explicit family of compatible functions
$H^1(\mathcal{N}(\mathcal{U});\Z) \longrightarrow [B,S^1]$
realizing the isomorphism from (\ref{eq:IsoCohoCirc}).
This is done in Theorem \ref{thm:ClassifyingForumlaHarmonic}, and an explicit
formula is given by   (\ref{eq:ClassifyingForumlaHarmonic}).

To begin, let $\{\varphi_j\}_{j\in J}$ be a partition of unity on $B$
dominated\footnote{ That is, so that $\mathsf{support}(\varphi_j) \subset \mathsf{closure}(U_j)$ for all $j\in J$.}
 by $\mathcal{U} = \{U_j\}_{j\in J}$,
fix a 1-cocycle $\eta = \{\eta_{jk}\} \in Z^1(\mathcal{N}(\mathcal{U});\Z)$,
and define for each $j\in J$ the map
\begin{equation}\label{eq:LiftCircCoords}
\begin{array}{rccl}
f_j :& U_j \times \{j\} \times \Z & \longrightarrow &\R  \\
& (b,j,n) & \mapsto & n + \sum\limits_{\ell} \varphi_\ell(b)\eta_{j\ell}
\end{array}
\end{equation}
Since   $\mathcal{U}$ is locally finite, then
all but finitely many terms in this sum  are zero.
Note that
$\Z$ acts on  $U_j\times \{j\}\times \Z$ by right translation $\big((b,j,n) , m\big) \mapsto (b,j, n + m)$,
and that $f_j$ is equivariant with respect to this action: $f_j(b,j,n+ m)  = f_j(b, j, n) + m$.
%

\noindent If $b\in U_j \cap U_k$, then we have that
\begin{eqnarray*}
  f_k(b, k, n + \eta_{jk}) &=& n + \sum_{\ell\in J} \varphi_\ell(b)(\eta_{k \ell} + \eta_{jk}) \\
   &=& n + \sum_{\ell\in J} \varphi_\ell(b)\eta_{j\ell} \\
   &=& f_j(b,j,n)
\end{eqnarray*}
and hence the $f_j$'s can be assembled  to induce  a continuous map $\gor{f}_\eta : E_\eta \longrightarrow \R$ on the quotient space
defined by (\ref{eq:BundleFromCocycle}); here $\eta = \{\eta_{jk}\}\in  Z^1(\mathcal{N}(\mathcal{U});\Z)$ is regarded as a collection of constant functions
$\eta_{jk} : U_j \cap U_k \longrightarrow \Z$.
To be more explicit, $\gor{f}_\eta$ sends the class of $(b,j,n)$ in $E_\eta$ to $f_j(b,j,n)\in \R$.
Since each $f_j$ is $\Z$-equivariant, then so is $\gor{f}_\eta$, and hence it descends to a
well defined map $f_\eta $ at the level of  base spaces
\begin{equation}\label{eq:ClassifyingForumlaInteger}
\begin{array}{cccc}
  f_\eta :&  B & \longrightarrow & S^1 \subset \C \\
   & U_j \ni b & \mapsto & \exp\left( 2\pi i   \sum\limits_k\varphi_k(b) \eta_{jk} \right)
\end{array}
\end{equation}

\begin{lemma}\label{prop:classifying}
The map $f_\eta$ classifies the principal  $\Z$-bundle $p_\eta : E_\eta \longrightarrow B$.
\end{lemma}
\begin{proof}
Let us see explicitly that the map $f_\eta$ is well defined; in other words, that the value $f_\eta(b)\in S^1$
is independent of the open set containing $b$.
Indeed, let $j,\ell \in J$ be so that $b\in U_j \cap U_\ell$.
We contend that $\varphi_k(b)\eta_{jk} = \varphi_k(b)(\eta_{j\ell} + \eta_{\ell k}  )$ for every $k\in J$.
If $b \notin U_k$, then the equality is trivial  since $\varphi_k(b) = 0$;
if $b \in U_k$, then $U_j \cap U_k \cap U_\ell \neq \emptyset$ and $\eta_{j k} = \eta_{j \ell } + \eta_{\ell k}$
since $\eta$ is a cocycle.
Therefore
\[
\sum\limits_k\varphi_k(b) \eta_{jk} = \eta_{j \ell} +  \sum\limits_k\varphi_k(b) \eta_{\ell k}
\]
and given that  $\eta_{j\ell} \in \Z$, then
$
\exp\left(2\pi i\sum\limits_k\varphi_k(b) \eta_{jk}\right) =
\exp\left(2\pi i\sum\limits_k\varphi_k(b) \eta_{\ell k}\right) $.

Finally, let us  check that taking  the pullback $f_\eta^*  \R $
of the universal $\Z$-bundle $\exp(2\pi i \; \cdot) :  \R \longrightarrow S^1$
yields a principal $\Z$-bundle  isomorphic to $p_\eta : E_\eta \longrightarrow B$.
Indeed, since $f_\eta \circ p_\eta  = \exp\left(2\pi i  \gor{f}_\eta\right) $,
then $\left(\gor{f}_\eta, f_\eta \right) : (p_\eta, E_\eta, B) \longrightarrow \left(\exp(2\pi i \; \cdot), \R , S^1\right)$
is a morphism of principal $\Z$-bundles, and the result follows from
\cite[Chapter 4: Theorem 4.2]{husemoller1966fibre}.
\end{proof}

\begin{theorem}\label{thm:ClassifyingForumlaHarmonic}
Let $\iota : \Z \hookrightarrow \R$ be the inclusion and
\begin{equation}\label{eq:realification}
\iota^* : H^1(\mathcal{N}(\mathcal{U}); \Z ) \longrightarrow H^1(\mathcal{N}(\mathcal{U}); \R )
\end{equation}
the induced homomorphism.
Given $\eta \in Z^1(\mathcal{N}(\mathcal{U}); \Z)$ and
$\tau\in C^0(\mathcal{N}(\mathcal{U}); \R)$, let
$\theta = \iota^\#(\eta) + \delta^0\tau $.
Denote by $\tau_j \in \R$  the value of $\tau$ on  the vertex  $j \in \mathcal{N}(\mathcal{U})$,
and by $\theta_{jk}\in \R$  the value of $\theta$ on the oriented edge $[j , k] \in  \mathcal{N}(\mathcal{U})$;
in particular $\theta_{jk} = - \theta_{kj}$, and $\theta_{jk} = 0$ whenever $\{j,k\} \notin \mathcal{N}(\mathcal{U})$.
If
\begin{equation}\label{eq:ClassifyingForumlaHarmonic}
\begin{array}{cccc}
  h_{\theta,\tau}: & B & \longrightarrow & S^1 \subset \C \\
   & U_j \ni b & \mapsto &  \exp\left\{ 2\pi i\left ( \tau_j +    \sum\limits_k \varphi_k(b) \theta_{jk}  \right) \right\}
\end{array}
\end{equation}
then $h_{\theta,\tau}$ is  a classifying map
for the principal $\Z$-bundle $p_\eta : E_\eta \longrightarrow B$.
\end{theorem}
\begin{proof} Since $f_\eta$ is a classifying map for $E_\eta$, by Lemma \ref{prop:classifying},
then  it is enough to check that $f_\eta$ and $h_{\theta,\tau}$ are homotopic (see Theorem \ref{thm:IsoHomoPrinG}).
For $b\in U_j$ we have that
\begin{eqnarray*}
  f_\eta(b) &=&   \exp\left( 2\pi i   \sum\limits_k \varphi_k(b) \eta_{jk} \right) \\
   &=& \exp\left( 2\pi i   \sum\limits_k \varphi_k(b) (\theta_{jk} + \tau_j - \tau_k) \right) \\
   &=&\exp\left( 2\pi i\left ( \tau_j +    \sum\limits_k \varphi_k(b) (\theta_{jk} -  \tau_k) \right) \right) \\
   &=& \nu_\tau(b)\cdot h_{\theta,\tau}(b)
\end{eqnarray*}
where $\nu_\tau(b) = \exp\left(-2\pi i \sum\limits_k\varphi_k(b)\tau_k\right)$.
Since $\nu_\tau$ factors through $\R$:
\[
\begin{array}{cccccc}
  \nu_\tau: & B & \longrightarrow & \R & \longrightarrow & S^1\subset \C \\
   & b & \mapsto & \sum\limits_k \varphi_k(b)\tau_k & \mapsto & \exp\left(-2\pi i \sum\limits_k\varphi_k(b)\tau_k\right)
\end{array}
\]
then $\nu_\tau$ is null-homotopic, hence  $f_\eta $ is homotopic to $ h_{\theta,\tau}$, and the result follows.
\end{proof}

\begin{remark}
  We note that the relation $ \theta =  \iota^\#(\eta)+ \delta^0\tau$ from Theorem \ref{thm:ClassifyingForumlaHarmonic}
   implies that the
  cochain $\tau \in C^0(\mathcal{N}(\mathcal{U}); \R)$ encodes the degrees of freedom  in choosing a
   cocycle representative for the class $\iota^*([\eta]) \in H^1(\mathcal{N}(\mathcal{U});\R)$,
   and thus defining the classifying map $h_{\theta, \tau} : B \longrightarrow S^1$.
  This choice will be addressed in the discussion about Harmonic Smoothing in Section \ref{sec:HarmonicSmoothing}.
\end{remark}

\section{Persistent Cohomology and Sparse Circular Coordinates for Data}\label{sec:RealData}
In this section we show how the theory we have developed thus far can be applied to real data sets.
In particular, we explain and justify the choices  made in the construction outlined in the Introduction (\ref{sec:SCCalgorithm}).
Let us begin by fixing an ambient metric space $(\M,\dd)$, let $L \subset \M$  be finite,
and let
\begin{eqnarray*}
B_\alpha(\ell) &=& \{b\in \M : \dd(b,\ell) < \alpha\} \;\;\;\;\; , \;\;\;\;\; \alpha \geq 0, \; \ell \in L \\[.1cm]
\mathcal{B}_\alpha &=& \{B_\alpha(\ell)\}_{\ell \in L} \\[.1cm]
L^{(\alpha)} &=& \bigcup \mathcal{B}_\alpha
\end{eqnarray*}
The formulas derived in the previous section, specially (\ref{eq:ClassifyingForumlaInteger}),
imply that   each cocycle  $\eta \in Z^1(\mathcal{N}(\mathcal{B}_\alpha); \Z)$
yields a map  $h: L^{(\alpha)} \longrightarrow S^1$.
The thing to notice is that $h$ is defined on every $b \in L^{(\alpha)}$;
thus, given a large but finite
set $X\subset \M$ --- the data --- sampled around a continuous space $\X \subset \M$, one can select a much smaller set of landmarks $L \subset X$ and $\alpha > 0 $
for which $X \subset L^{(\alpha)}$.
The resulting circular coordinates
$h : L^{(\alpha)} \longrightarrow S^1$
will thus be defined on all points of $X$, though only
the landmark set is used in its construction.
As we alluded to in the introduction, this is what we mean  when we say that the coordinates are \emph{sparse}.

\subsection*{Landmark Selection}
In practice we select the landmarks $L\subset X$ either at random, or through \verb"maxmin" sampling:
Given $N\leq |X|$ and $\ell_1 \in X$ chosen arbitrarily,
assume that $\ell_1, \ldots, \ell_j\in X$ have been selected,  $1 \leq j < N$,
and let
\begin{equation}\label{eq:maxMinSampling}
\ell_{j+1} = \argmax_{x\in X} \min \big\{ \dd(x, \ell_1) ,\ldots, \dd(x, \ell_j) \big\}
\end{equation}
Following this inductive procedure defines a landmark set $L = \{\ell_1,\ldots, \ell_N\} \subset X$
that is in practice well-separated and well-distributed throughout the data.
However, it is important to keep in mind that this process is
 prone to choosing outliers.

\subsection*{The Subordinated Partition of Unity}
As for the choice of partition of unity $\{\varphi_\ell\}_{\ell \in L}$ dominated by $\mathcal{B}_\alpha$,
we can use that the cover is via metric balls,  and let
\begin{equation}\label{eq:PartitionOfUnity}
\varphi_\ell(b) =
\frac{|\alpha - \dd(\ell, b)|_+}{ \sum\limits_{\ell' \in L} |\alpha - \dd(\ell',b)|_+}
\;\;\;\;\; \;\; \mbox{ where }\;\;\;\;\;\;\;
|r|_+ = \max\{r,0\}, \;\; r\in \R
\end{equation}
See \cite[3.3 and Fig 6.]{perea2018multiscale}
for other typical choices of partition of unity in the case of metric spaces, and coverings via metric balls.

\subsection*{The Need for Persistence}
Even if the landmark set $L$ correctly approximates the underlying topology of
$X$, the choice of scale $\alpha >0$ and cocycle $\eta \in Z^1(\mathcal{N}(\mathcal{B}_\alpha);\Z)$
might reflect sampling artifacts instead of robust geometric features of the underlying space $\X$.
This is why we need persistent cohomology.
Indeed, a class $[\eta] \in H^1(\mathcal{N}(\mathcal{B}_\alpha) ; \Z)$
which is not in the kernel of the homomorphism
\[
H^1(\mathcal{N}(\mathcal{B}_{\alpha}) ; \Z)  \longrightarrow H^1(\mathcal{N}(\mathcal{B}_{\alpha'}); \Z)
\;\;\;\;\; , \;\;\;\;\;
0< \alpha' < \alpha
\]
induced by the inclusion
$\mathcal{N}(\mathcal{B}_{\alpha'})  \subset \mathcal{N}(\mathcal{B}_\alpha) $,
 is less likely to correspond to spurious features as the difference $\alpha - \alpha'$ increases.
Note, however, that the efficient computation of persistent cohomology classes relies on
using field coefficients.
We proceed, following \cite{de2011persistent} and \cite{perea2018multiscale},
by choosing  a prime $q>2$ and a scale $\alpha > 0$ so that
(1)   $H^1(\mathcal{N}(\mathcal{B}_{\alpha}) ; \Z/q)$ contains a class with large persistence,
and (2) so that the homomorphism
$H^1(\mathcal{N}(\mathcal{B}_{\alpha}) ; \Z)  \longrightarrow H^1(\mathcal{N}(\mathcal{B}_{\alpha}); \Z/q) $,
induced by the quotient map $\Z \longrightarrow \Z/q$, is surjective.

\subsection*{Lifting Persistence  to Integer Coefficients}
As stated in \cite{de2011persistent}, one has  that:
\begin{proposition}\label{prop:IntegerLift}
Let $K$ be a finite simplicial complex, and
suppose that $q \in \N$ does not divide the order of the torsion
subgroup of
$H^2(K;\Z)$.
Then the homomorphism
\[
\iota^*_q :
H^1(K;\Z) \longrightarrow H^1(K;\Z/q)
\]
induced by the quotient map $\iota_q : \Z \longrightarrow \Z/q$,
is surjective.
\end{proposition}
\begin{proof}
This follows directly from  the Bockstein long exact sequence in cohomology,
corresponding to the short exact sequence
$
\begin{tikzcd}[column sep = scriptsize]
0 \ar{r}& \Z \ar[r, "\times q"] & \Z \ar[r, "\iota_q"] & \Z/q \ar{r} & 0.
\end{tikzcd}
$
\end{proof}

More generally, let $\{K_\alpha\}_{\alpha\geq 0 }$ be a filtered simplicial complex
with $\bigcup\limits_{\alpha \geq 0} K_\alpha$ finite.
Since each complex $K_\alpha$ is finite,
and the cohomology groups $H^2(K_\alpha; \Z)$ change  only  at finitely many values of $\alpha$,
then there exists $Q\in \N$ so that  the hypotheses of Proposition \ref{prop:IntegerLift}
will be satisfied for each  $q\geq Q$, and all $\alpha\geq 0$.
In practice we choose a prime $q$ at random, with the intuition that for scientific data
only a few primes are  torsion contributors.

Let $\Z/q = \{0,1, \ldots, q-1\}$ and for $\eta' \in Z^1(K_\alpha; \Z/q)$ let   $\eta\in C^1(K_\alpha;\Z)$
be defined on each 1-simplex $\sigma\in K_\alpha$ as:
\begin{equation}\label{eq:CoycleLift}
\eta(\sigma) =
\left\{
  \begin{array}{lcr}
    \eta'(\sigma) & \hbox{if } & \eta'(\sigma) \leq \frac{q-1}{2}  \\[.25cm]
    \eta'(\sigma) - q & \hbox{ if }& \eta'(\sigma) > \frac{q-1}{2}
  \end{array}
\right.
\end{equation}
Thus,  $\eta$ takes values in $ \left\{-\frac{q-1}{2} , \ldots,  0 ,  \ldots, \frac{q-1}{2}\right\} \subset \Z$
and it satisfies $(\eta \mod q) = \eta'$.
For the examples we have observed, the cochain defined by (\ref{eq:CoycleLift}) produces an integer cocycle.
One of the reviewers of an earlier version of this paper remarked that this is not always the case;
the outlined procedure tends to fail (in real world-examples)
when the cohomology computation involves division by 2.
As highlighted in \cite[2.4]{de2011persistent},    solving a Diophantine linear system can be used to fix the problem.

\subsection*{Use Rips, not Nerves}
Constructing the filtered complex $\{ \mathcal{N}(\mathcal{B}_\alpha)\}_{\alpha\geq 0}$
can be rather expensive  for a general ambient metric space $(\M,\dd)$.
Indeed, the inclusion of an $n$-simplex into the nerve complex is predicated on checking if the intersection of
$n+1 $ ambient metric balls is nonempty. This is nontrivial on curved spaces.
On the other hand, the Rips complex
\[
R_\alpha(L) = \{\sigma \subset L : \mathsf{diam}(\sigma) < \alpha\}  \;\;\;\; , \;\;\;\; \alpha \geq 0
\]
provides  a   straightforward alternative, since we can  use  that
\[R_\alpha(L) \subset \mathcal{N}(\mathcal{B}_\alpha) \subset R_{2\alpha}(L)
\]
for every $\alpha\geq0$.
Here is how. Let $q > 2$ be a prime so that
\[\iota_q^*:
H^1(R_\alpha(L) ; \Z) \longrightarrow H^1(R_\alpha(L);\Z/q)
\]
is surjective for all $\alpha \geq 0$,
 and let
\[
\jmath : H^1(R_{2\alpha}(L);\Z/q) \longrightarrow H^1(R_\alpha(L) ; \Z/q)
\]
be the homomorphism induced by the inclusion
$R_\alpha(L) \subset R_{2\alpha}(L)$.
Moreover, let
$\eta' \in Z^1(R_{2\alpha}(L);\Z/q) $
be so that $[\eta'] \notin \mathsf{Ker}(\jmath)$,
and fix  an integer lift
\[
\eta \in Z^1(R_{2\alpha}(L);\Z)
\]
That is, one
for which
$
\eta' - (\eta \mod q) \in Z^1(R_{2\alpha}(L) ; \Z)
$ is a coboundary, e.g. (\ref{eq:CoycleLift}).

The   diagram below summarizes the spaces and homomorphisms used thus far:
\[
\begin{tikzcd}
\hspace{-0.9cm}\left[\eta'\right]\in
H^1(R_{2\alpha}(L);\Z/q)
\arrow[d]
\arrow[dd, bend right=75, "\jmath"']
& H^1(R_{2\alpha}(L);\Z) \ni \left[\eta\right] \hspace{-0.8cm}
\arrow[l, "\iota^*_q"']
\arrow[d, "\iota^*_\Z"]
\\
H^1(\mathcal{N}(\mathcal{B}_{\alpha});\Z/q)
\arrow[d]
& H^1(\mathcal{N}(\mathcal{B}_\alpha);\Z) \ni \left[ \gor{\eta}\right] \hspace{-0.8cm}
\arrow[l, "\iota^*_q"]
\\
H^1(R_{\alpha}(L);\Z/q)
\end{tikzcd}
\]
Since the diagram commutes,
then  $\left[\eta\right]$ is not in the kernel of $\iota_\Z^*$,
and hence we obtain a nonzero element
$\iota_\Z^*([\eta]) = [\gor{\eta}] \in H^1(\mathcal{N}(\mathcal{B}_\alpha);\Z)$.
This is the class  we would use as input for  Theorem \ref{thm:ClassifyingForumlaHarmonic}.

\subsection*{Harmonic Smoothing}\label{sec:HarmonicSmoothing}
The final step is selecting an appropriate   cocycle representative (refer to Figure \ref{fig:example_circle} to see why this matters)
\[
\gor{\theta} \in Z^1 (\mathcal{N}(\mathcal{B}_\alpha ); \R)
\]
for the class $\iota^*([\gor{\eta}]) \in H^1(\mathcal{N}(\mathcal{B}_\alpha);\R)$,
see (\ref{eq:realification}).
Again, since one would hope to never compute the nerve complex, the strategy is to solve the problem
in $Z^1(R_{2\alpha}(L) ;\R)$ for $\iota^\#(\eta)$,
and then transfer the solution
using $\iota^\#_\R : C^1(R_{2\alpha}(L); \R)  \rightarrow C^1(\mathcal{N}(\mathcal{B}_\alpha); \R)$.

Inspecting (\ref{eq:ClassifyingForumlaHarmonic}) reveals that
the choice of $\gor{\theta}$ which promotes the smallest total variation in $h_{\gor{\theta}, \gor{\tau}}$,
is the one for which the value of $\gor{\theta}$  on each 1-simplex
of $\mathcal{N}(\mathcal{B}_\alpha)$
is as small as possible.
Consequently, we will look for the cocycle representative
\[
\theta \in Z^1(R_{2\alpha}(L);\R)
\]
of $\iota^*([\eta])$, which in average has the smallest squared value\footnote{ That is, we use the harmonic cocycle representative for appropriate  inner products on cochains} on each
1-simplex of $R_{2\alpha}(L)$.
That said, not all edges in the rips complex $R_{\epsilon}(L)$ are created equal.
Some might have just entered the filtration, i.e. $\dd(\ell_j, \ell_k) \approx  \epsilon$,
which would make them unstable if $L$ is corrupted with noise,
or perhaps $X \cap \left(B_{\epsilon/2}(\ell_j) \cup B_{\epsilon/2}(\ell_k)\right)$ is a rather small
portion of the data, which could happen if $\ell_j $ and $\ell_k$ are outliers selected during \verb"maxmin" sampling.

These observations can be encoded by  choosing weights on  vertices and edges:
\begin{equation}\label{eq:Weights}
\omega_\epsilon : L \times L \longrightarrow [0, \infty) \;\; , \;\; \epsilon \geq 0
\end{equation}
where $\omega_\epsilon$ is symmetric for all $\epsilon > 0$, it satisfies
\[
\omega_{\epsilon'} (\ell , \ell') \leq \omega_\epsilon(\ell, \ell') \;\;\;\;, \;\;\;\; \mbox{ for } \;\; \epsilon' \leq \epsilon
\]
and   $\omega_\epsilon(\ell, \ell') = 0$
only when    $ 0<\epsilon \leq \dd(\ell, \ell')$.
Here $\omega_\epsilon(\ell,\ell)$ is the weight of $\ell$, and $\omega_{\epsilon}(\ell,\ell')$ is the weight
of the edge $\{\ell,\ell'\}$.
For instance, one can take
\[
\omega_{\epsilon}(\ell,\ell') =   |\epsilon - \dd(\ell,\ell')|_+
\]
but we note that we have not yet systematically investigated the effects of this choice.
See \cite[Apendix D]{rybakken2017decoding} for a different heuristic.

It follows that $\omega_\epsilon$ defines inner products $\< \cdot , \cdot \>_\epsilon$ on $C^0(R_\epsilon(L) ; \mathbb{R})$
and $C^1(R_\epsilon(L) ; \mathbb{R})$,
by letting the indicator functions $1_\sigma$ on $k$-simplices ($k=0,1$) $\sigma \in R_\epsilon(L)$ be orthogonal, and setting
\begin{equation}\label{eq:InnerProduct}
\left \<
1_{ \sigma } , 1_{\sigma}
\right \>_\epsilon
=
\omega_\epsilon(\sigma)
\end{equation}
Using   $\<\cdot,\cdot\>_{\epsilon}$, for $\epsilon =2\alpha$,
we let $\beta\in B^1(R_{2\alpha}(L);\R)$ be the orthogonal projection of $\iota^\#(\eta)$
onto the space of 1-coboundaries,
and define
\begin{equation}\label{eq:FormulaHarmonicCocyle}
\theta = \iota^\#(\eta) - \beta
\end{equation}
A bit of linear algebra shows that,
\begin{proposition}
The 1-cocycle $\theta$ defined by (\ref{eq:FormulaHarmonicCocyle}) is a minimizer for the weighted least squares problem
\begin{equation}\label{eq:LeastSquares}
\min_{\phi \sim \iota^\#(\eta)} \;\;
\sum_{\sigma }
\omega_{2\alpha}(\sigma) \cdot \phi(\sigma)^2
\end{equation}
Here the sum runs over all 1-simplices $\sigma \in R_{2\alpha}(L)$,
and the minimization is over all 1-cocycles $\phi \in Z^1(R_{2\alpha}(L); \R)$ which are cohomologous to $\iota^\#(\eta)$.
\end{proposition}

Similarly, and if
\[
d_{2\alpha} : C^0(R_{2\alpha}(L); \R) \longrightarrow C^1(R_{2\alpha}(L);\R)
\]
denotes the coboundary map, then we let
\[
\tau \in \mathsf{Ker}(d_{2\alpha})^\perp
\subset
C^0(R_{2\alpha}(L);\R)
\]
in the orthogonal complement of the kernel of $d_{2\alpha}$,
be so that $d_{2\alpha}(\tau) = -\beta$.
Hence  $\tau$ is the 0-chain with the smallest norm mapping to $-\beta$ via $d_{2\alpha}$.
Consequently, if
\[
d_{2\alpha}^+ : C^1(R_{2\alpha}(L) ;\R) \longrightarrow C^0(R_{2\alpha}(L); \R)
\]
is the weighted Moore-Penrose pseudoinverse of $d_{2\alpha}$ (see \cite[III.3.4]{ben2003generalized}), then
\begin{equation}\label{eq:MP_harmonic}
\tau = - d_{2\alpha}^+(\iota^\#(\eta))
\hspace{1cm}\mbox{ and } \hspace{1cm}
\theta =
\iota^\#(\eta) \, +\, d_{2\alpha}\left(\tau\right)
\end{equation}
This is how we compute $\tau$ and $\theta$ in our implementation.
Now, let
\begin{eqnarray*}
\gor{\tau}  = \iota_\R^\#(\tau) &\in& C^0(\mathcal{N}(\mathcal{B}_\alpha); \R) \\[.1cm]
\gor{\theta} = \iota_\R^\#(\theta) &\in& Z^1(\mathcal{N}(\mathcal{B}_\alpha); \R)
\end{eqnarray*}

If we were to be completely rigourous, then $\gor{\tau}$ and $\gor{\theta}$ would be the
cochains going into (\ref{eq:ClassifyingForumlaHarmonic});
this would require the 1-skeleton of the nerve complex.
However, as the following proposition shows, this is unnecessary:
\begin{proposition}\label{prop:unnecesary}
For all $b\in B_{\alpha}(\ell_j)$, and every $j =1,\ldots, N$,  we have that
\[
 \exp\left\{ 2\pi i\left ( \gor{\tau}_j +    \sum\limits_{k=1}^N \varphi_k(b) \gor{\theta}_{jk}  \right) \right\}
=
 \exp\left\{ 2\pi i\left ( \tau_j +    \sum\limits_{k=1}^N \varphi_k(b) \theta_{jk}  \right) \right\}
\]
That is, we can compute sparse  circular coordinates using only the Rips filtration
on the landmark set.
\end{proposition}
\begin{proof}
Since $\mathcal{N}(\mathcal{B}_\alpha)$ and $R_{2\alpha}(L)$ have the same vertex set,
namely $L$,
then $\gor{\tau} = \tau$ as real-valued functions on $L$.
Moreover, for all $k =1,\ldots, N$ we have that
\[
\varphi_k(b)\gor{\theta}_{jk} = \varphi_k(b)\theta_{jk}
\]
for if $b \notin B_\alpha(\ell_k)$, then both sides are zero, and if $b\in B_{\alpha}(\ell_j) \cap B_{\alpha}(\ell_k)$,
then the edge $\{\ell_j, \ell_k\}$ is in both $R_{2\alpha}(L)$ and $\mathcal{N}(\mathcal{B}_\alpha)$, which shows
that $\gor{\theta}_{jk} = \theta_{jk}$.
\end{proof}

\section{Experiments}\label{sec:Experiments}
In all experiments below,
persistent cohomology is computed using a MATLAB wrapper for Ripser \cite{bauer2017ripser} kindly provided by \href{http://www.ctralie.com/}{Chris Tralie}.
The Moore-Penrose pseudoinverse was computed via MATLAB's \verb"pinv".
In all cases we
run the algorithm from the Introduction \ref{sec:SCCalgorithm} using
the  indicated  persistence classes, or linear combinations thereof as made explicit in each example.

\subsection{Synthetic Data}
\subsubsection{A Noisy Circle} We select 1,000 points from a noisy circle in $\R^2$;
the noise is Gaussian in the direction normal  to the unit circle.
50 landmarks  were selected via \verb"maxmin" sampling (5\% of the data),
and circular coordinates were computed for the two most persistent classes $\eta_1$ and $\eta_2$,
using  (\ref{eq:MP_harmonic}) as input to (\ref{eq:ClassifyingForumlaHarmonic})  --- this is the harmonic cocycle column ---
or  (\ref{eq:ClassifyingForumlaInteger}) with either $\eta_1$ or $\eta_2$ directly --- the integer cocycle column.
We show the results in Figure \ref{fig:example_circle} below.
Computing persistent cohomology took  0.079423 seconds (the Rips filtration is constructed   from zero to the diameter of the landmark set); in each case computing the
harmonic cocycle takes about  0.037294 seconds.
This example highlights the inadequacy of the integer cocycle and of choosing cohomology classes
associated to sampling artifacts (i.e., with low persistence).
From now on, we only present circular coordinates computed with the relevant harmonic cocycle representative.
\begin{figure}[!htb]
  \centering
      \includegraphics[width=\textwidth]{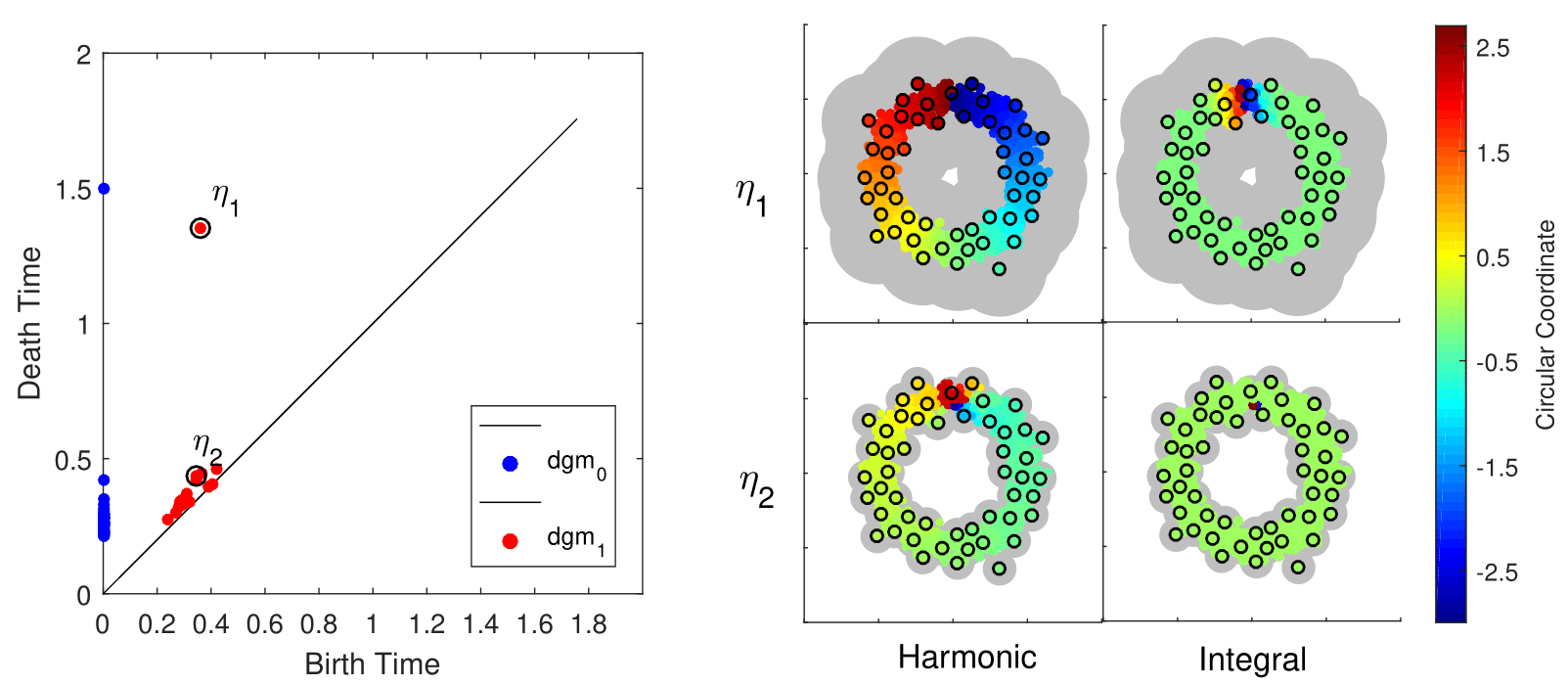}
  \caption{A noisy circle. Left: persistence diagrams in dimension 0  (blue) and 1 (red)
for the Rips filtration on the landmarks. Right: Circular coordinates from the two most persistent
classes $\eta_1$ (top row) and $\eta_2$ (bottom row). The columns indicate if the harmonic or integral cocycle was used.
The dark rings are the  landmarks.
The colors are: the domain of  definition for the circular coordinate (gray), and its value on each point
(dark blue, $-\pi$,  through dark red, $\pi$).
Please refer to an electronic version for colors.}
  \label{fig:example_circle}
\end{figure}

\subsubsection{The 2-Dimensional Torus}
For this experiment we sample 1,000 points uniformly at random from the square $[0, 2\pi)\times [0,2\pi)$,
and for each selected pair $(\phi_1,\phi_2)$ we generate a point
$\left(e^{i\phi_1} , e^{i\phi_2}\right) \in S^1\times S^1 $ on the surface of the torus embedded in $\C^2$.
The resulting finite set is endowed with the ambient distance from $\C^2$,
and 100 landmarks (i.e., 10\% of the data) are selected through \verb"maxmin" sampling.
We show the results in Figure \ref{fig:example_torus} below, for the circular coordinates
computed with the two most persistent classes, $\eta_1$ and $\eta_2$, and the maps (\ref{eq:ClassifyingForumlaHarmonic})
associated to the harmonic cocycle representatives (\ref{eq:MP_harmonic}).
Computing persistent cohomology for the Rips filtration on the Landmarks (from zero to the diameter of the set)
takes 0.398252 seconds, and computing the harmonic cocycles takes  0.030832 seconds.
\begin{figure}[htb!]
  \centering
  \includegraphics[width=\textwidth]{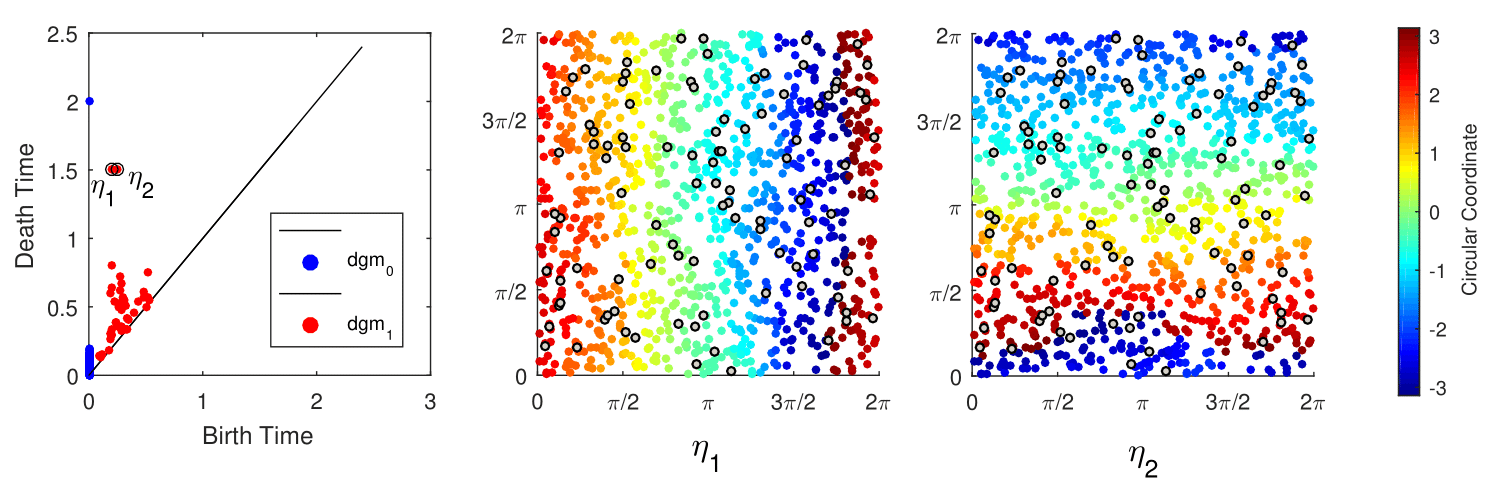}
  \caption{The torus. Left: Persistence in dimensions 0 and 1 for the Rips filtration on the landmark set.
Center and Right: the landmark set is depicted with dark rings, and the colors correspond to the  circular coordinates
computed with (the harmonic representatives from) the two most persistent classes $\eta_1$ (center) and $\eta_2$ (right).
Please refer to an electronic version for colors.}
  \label{fig:example_torus}
\end{figure}

\subsubsection{The Klein Bottle} We model the Klein bottle as the quotient space
\[
K=
S^1\times S^1 / (z,w) \sim (-z ,\bar{w})
\]
and endow it with the quotient metric.
Just like in the case of the 2-torus, we sample 1,000 points uniformly at random on (the fundamental domain $[0,\pi)\times [0,2\pi)$ of) $K$,
 and select 100 landmarks
via \verb"maxmin" sampling and the quotient metric.
Below in Figure \ref{fig:example_klein} we show the results of computing the persistent cohomology, with coefficients in $\Z/13$,  of the Rips filtration
on the landmark set (left), along with the  circular coordinates corresponding to the most persistent class (right).
\begin{figure}[!htb]
  \centering
  \includegraphics[width=.8\textwidth]{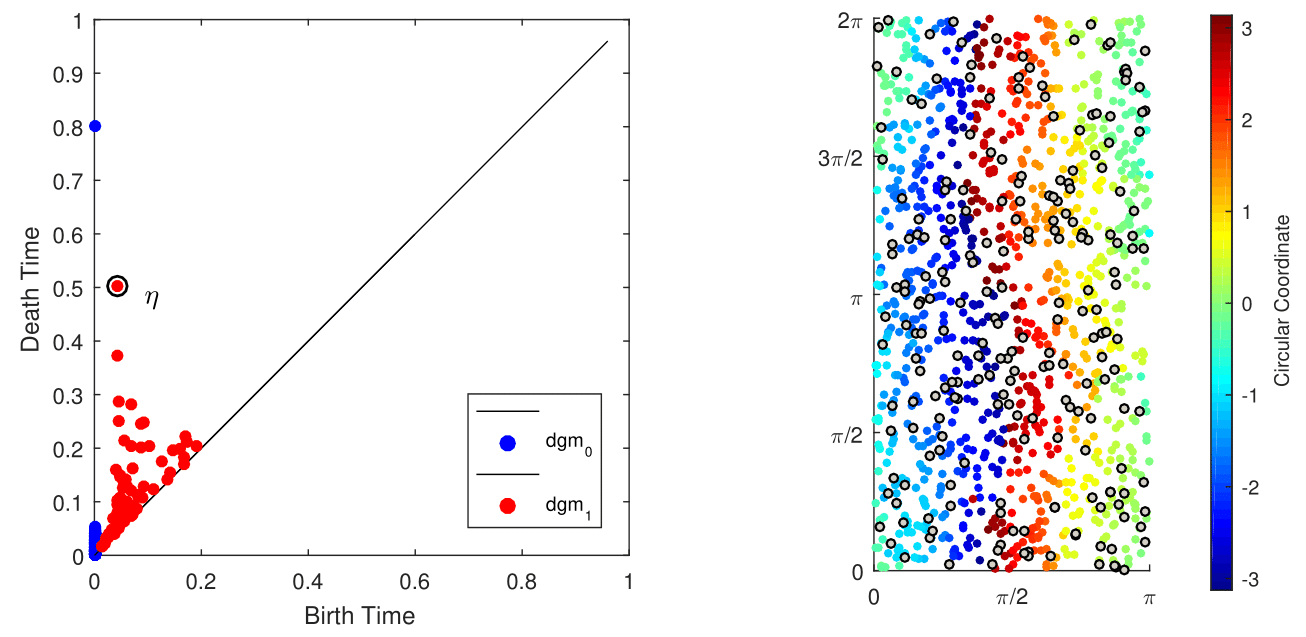}
  \caption{Circular coordinates on the Klein bottle. Left: Persistence with coefficients in $\Z/13$ for the Rips filtration on the landmark set.
Right: Circular coordinates computed from the harmonic representative from  the class $\eta$ with largest persistence.
Dark rings indicate landmarks, and the colors (dark blue through dark red) are the angular values of the circular coordinate on each data point.
Please refer to an electronic version for colors. }\label{fig:example_klein}
\end{figure}

\subsection{Real Data}
\subsubsection{COIL 20}
The Columbia University Image Library (COIL-20) is a collection of $448\times 416$-pixel gray scale images
from 20 objects, each of which is photographed at  72 different rotation angles \cite{nene1996columbia}.
The database has two versions: a processed version, where the images have been cropped to show only the rotated object,
and an unprocessed version with the 72 raw images from 5 objects. We will analyze the unprocessed database,
of which a few examples are shown in Figure \ref{fig:COIL20} below.

\begin{figure}[!htb]
  \centering
  \includegraphics[width=0.6\textwidth]{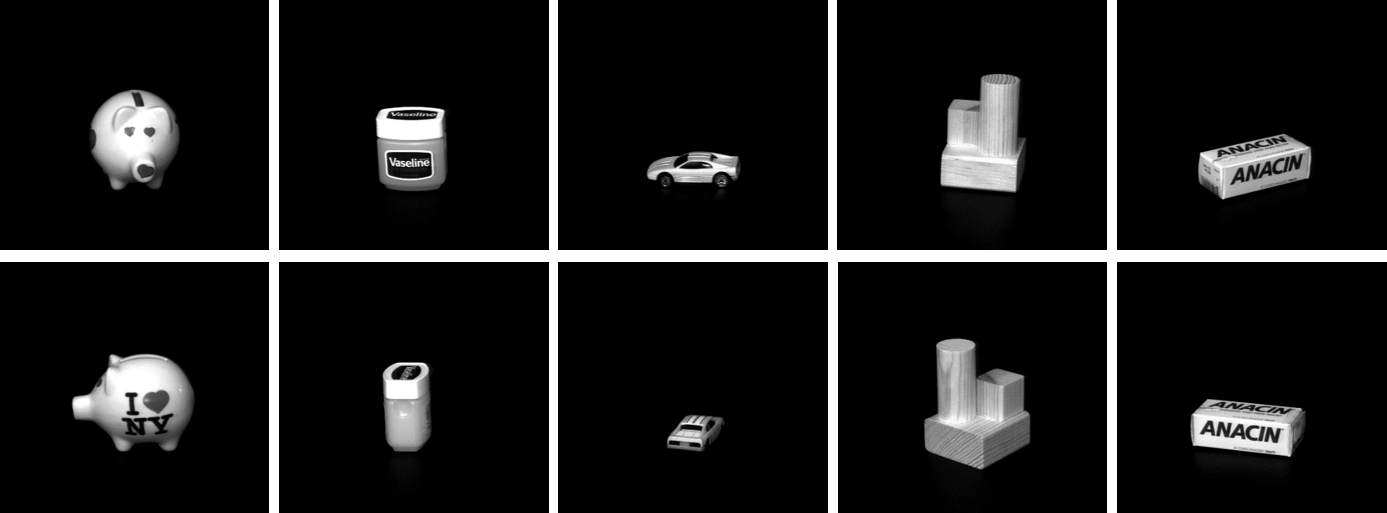}
  \caption{Some examples from the unprocessed COIL20 image database}\label{fig:COIL20}
\end{figure}

Regarding each image as a vector of pixel intensities in $\R^{448\times 416}$
yields a  set $X$ with 360 points; this set becomes a finite metric space when endowed with the ambient Euclidean distance.
Below in Figure \ref{fig:coil20_iso_circular} (left) we show the result of computing persistence
(this time visualized as barcodes) for the Rips complex on the entire data set (0.293412 seconds).
Each one of the six most persistent classes $\eta_1,\ldots, \eta_6$ yields
a circle-valued map on the data $h_j : X \longrightarrow S^1$, $j=1,\ldots, 6$.
Multiplying these maps together, using the group structure from $S^1 \subset \C$,
yields a map $h : X \longrightarrow S^1$. We do this at the level of maps, as opposed to adding up the cocycle representatives,
because there is no scale $\alpha$ at which all these classes are alive.
We also show in Figure \ref{fig:coil20_iso_circular} (right)
an Isomap \cite{tenenbaum2000global} projection of the data onto $\R^2$, and we color each projected data point with its $h$ value.

\begin{figure}[!htb]
  \centering
  \includegraphics[width=\textwidth]{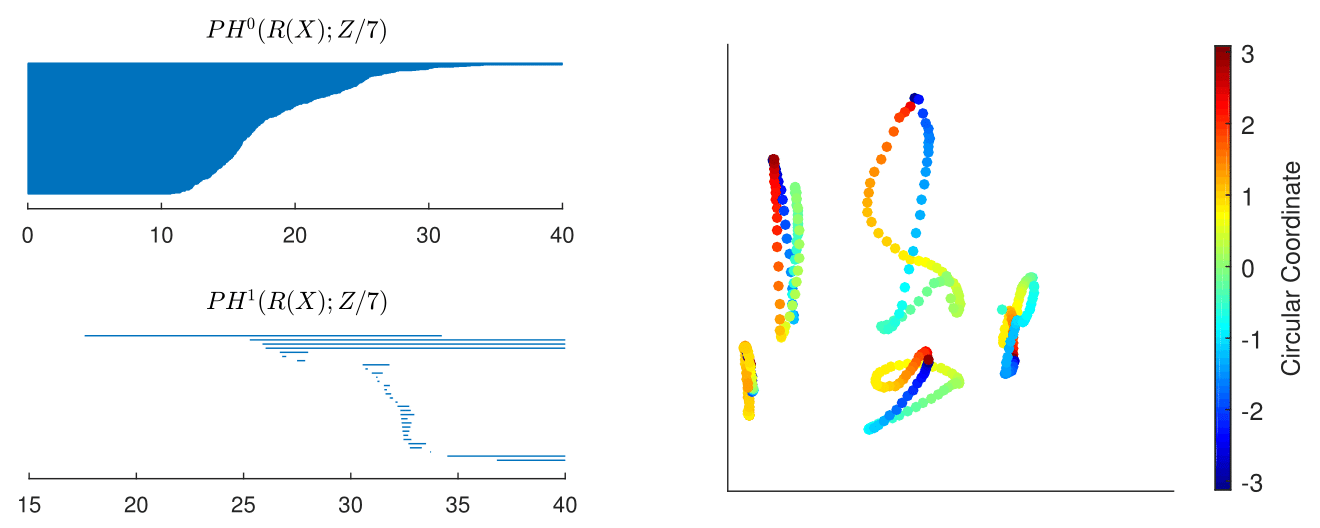}
  \caption{COIL-20 Unprocessed. Left: Persistence of the Rips filtration. Right: Isomap projection colored by circular coordinate.}
  \label{fig:coil20_iso_circular}
\end{figure}

As we show in Figure \ref{fig:COIL20ClustCirc} below, a better system of coordinates for the data (i.e. one without crossings)
is given by the computed circular coordinate of each data point, and the
cluster (computed using single linkage) to which it belongs to.

\begin{figure}[!htb]
  \centering
  \includegraphics[width=0.8\textwidth]{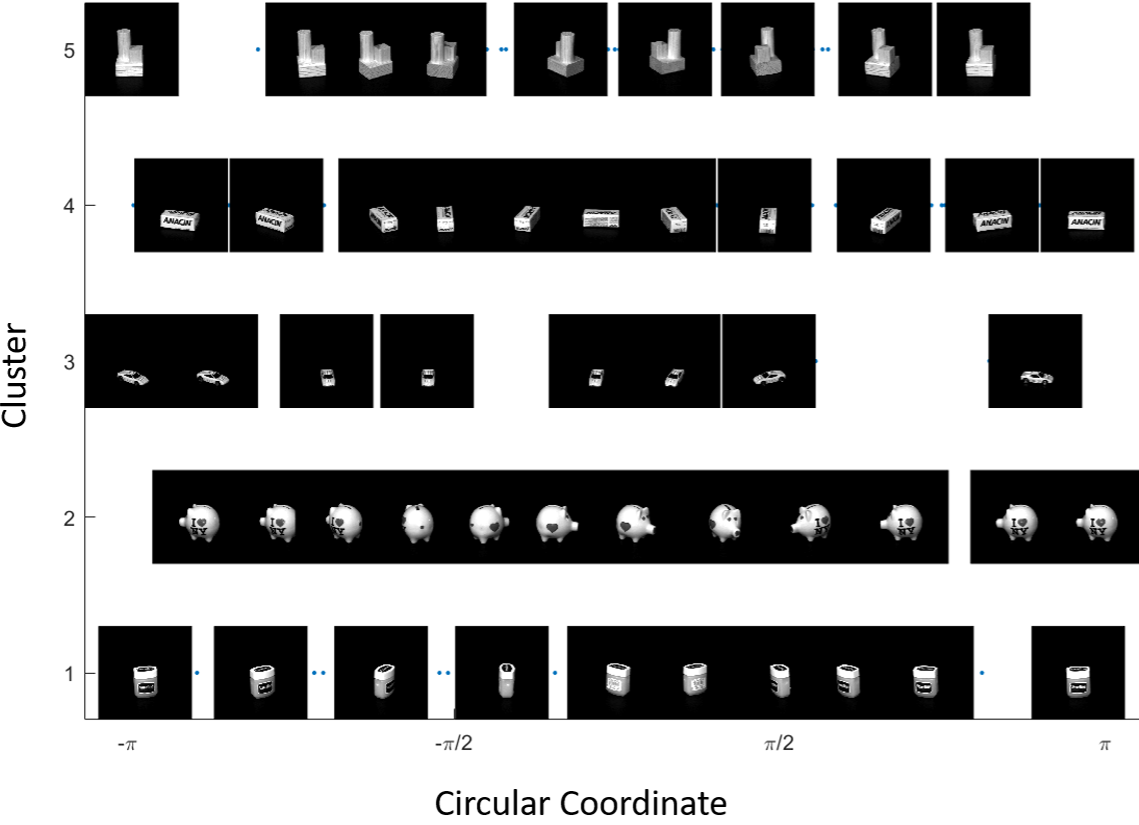}
  \caption{COIL-20 Unprocessed: Clusters vs Circular coordinates}
  \label{fig:COIL20ClustCirc}
\end{figure}

\subsubsection{The Mumford Data}
This data set was first introduced in \cite{lee2003nonlinear},
with an initial exploration  of its underlying topological structure done in \cite{de2004topological},
and then a more thorough investigation in  \cite{carlsson2008local}.
The data set in question is a collection of roughly 4 million $3\times 3$-pixel
gray-scale  images with high-contrast, selected from monochrome photos in a database of 4,000 natural scenes \cite{hateren_schaaf_1998}.
The $3\times 3$-pixel image patches are preprocessed, intensity-centered and contrast-normalized, and a linear change of coordinates is performed yielding a point-cloud
$\mathcal{M} \subset S^7 \subset \R^8$.
The Euclidean distance in $\R^8$ endows $\mathcal{M}$ with the structure of a finite metric space.
Following \cite{carlsson2008local}, we select $50,000$ points at random from $\mathcal{M}$
and then let $X$ be the top $30\%$ densest points as measured  by the distance to their 15th nearest neighbor.
This results in a data set with 15,000 points, which we analyze below.

We select 700 landmarks from $X$ via \verb"maxmin" sampling, i.e. 4.7\% of $X$,
and compute  persistence for the associated Rips filtration.
This takes about 2.2799 seconds and the result is shown in Figure \ref{fig:three_circ_Barcodes}.
\begin{figure}[htb!]
  \centering
  \includegraphics[width=.9\textwidth]{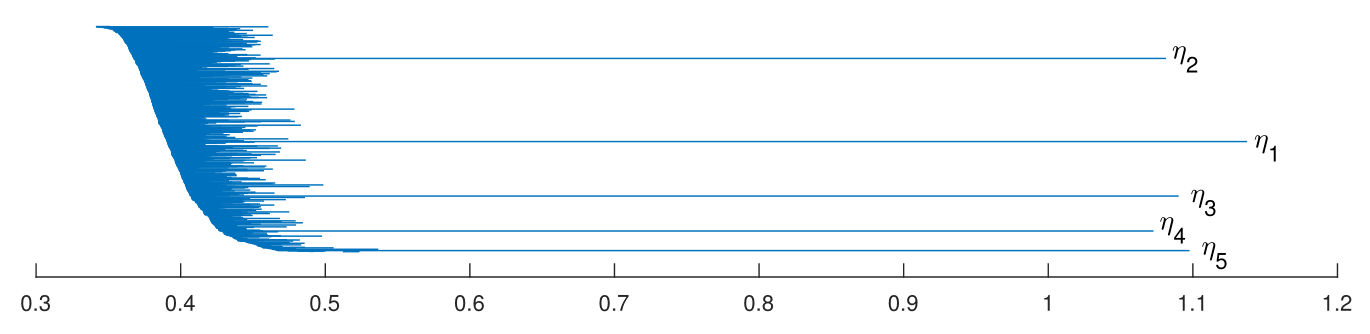}
  \caption{Barcodes from persistence on the Rips filtration of the landmark set $L\subset X$.}\label{fig:three_circ_Barcodes}
\end{figure}

\noindent Each bar in the barcode yields a class $\eta_j$, which we order from largest ($\eta_1$) to smallest  ($\eta_5$) persistence.
Below in Figure \ref{fig:three_circ_colors} we show the circular coordinates
associated to the classes $\eta_2$, $\eta_1 + \eta_5$ and $\eta_3 + \eta_4$, respectively.
Each of the three panels shows a scatter plot of $X \subset \R^8$ with respect to the first two coordinates,
dark rings are the selected landmarks, and the colors (dark blue through dark red) are the circular coordinates
corresponding to the indicated persistence classes. The computation of each cocycle representative takes about 7.1434 seconds, so
the entire analysis is less than 25 seconds.
\begin{figure}[htb!]
  \centering
  \includegraphics[width=0.9\textwidth]{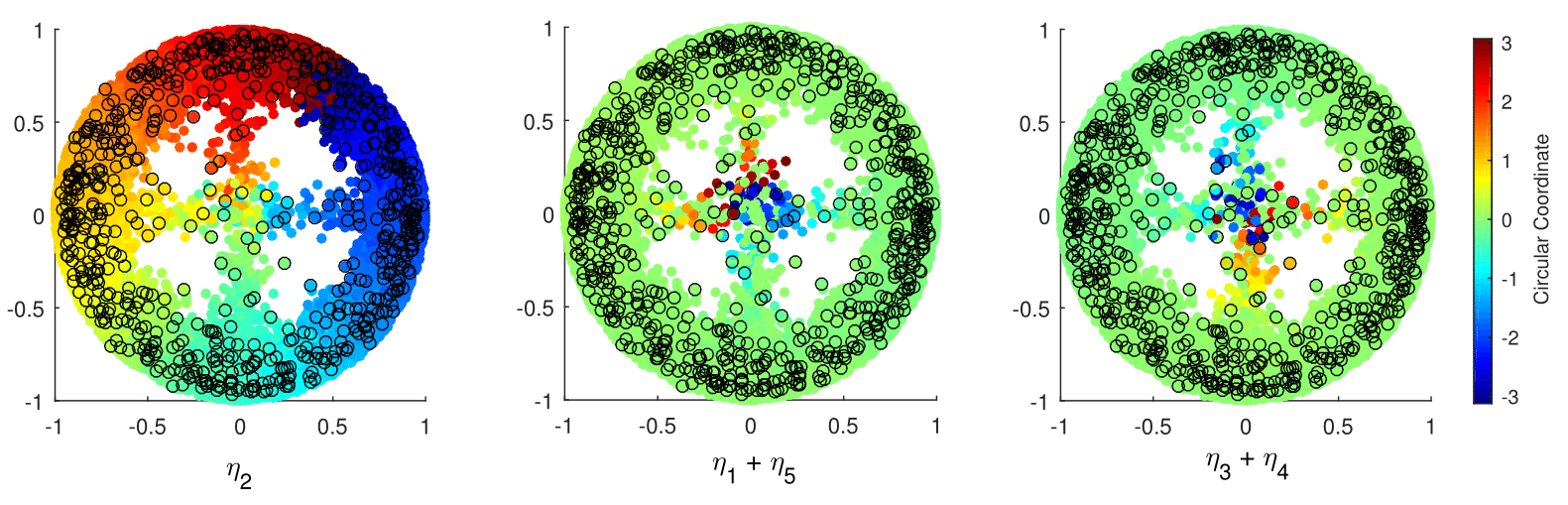}
  \caption{Circular coordinates for the points in $X\subset \R^8$, plotted  according to their first two coordinates, and colored by
the circular coordinates associated to each one of the classes $\eta_2$ (left), $\eta_1 + \eta_5$ (center) and $\eta_3 + \eta_4$ (right).}
  \label{fig:three_circ_colors}
\end{figure}

These three circular coordinates allow us to map the data set $X$ into the 3-dimensional torus $T^3 = S^1\times S^1\times S^1$,
which we model as the 3-dimensional cube $[-\pi, \pi] \times [-\pi, \pi] \times [-\pi , \pi]$ with opposite faces identified.
We show in Figure \ref{fig:three_circ_coords} below the result of mapping the data into $T^3$.
\begin{figure}[htb!]
  \centering
  \includegraphics[width=\textwidth]{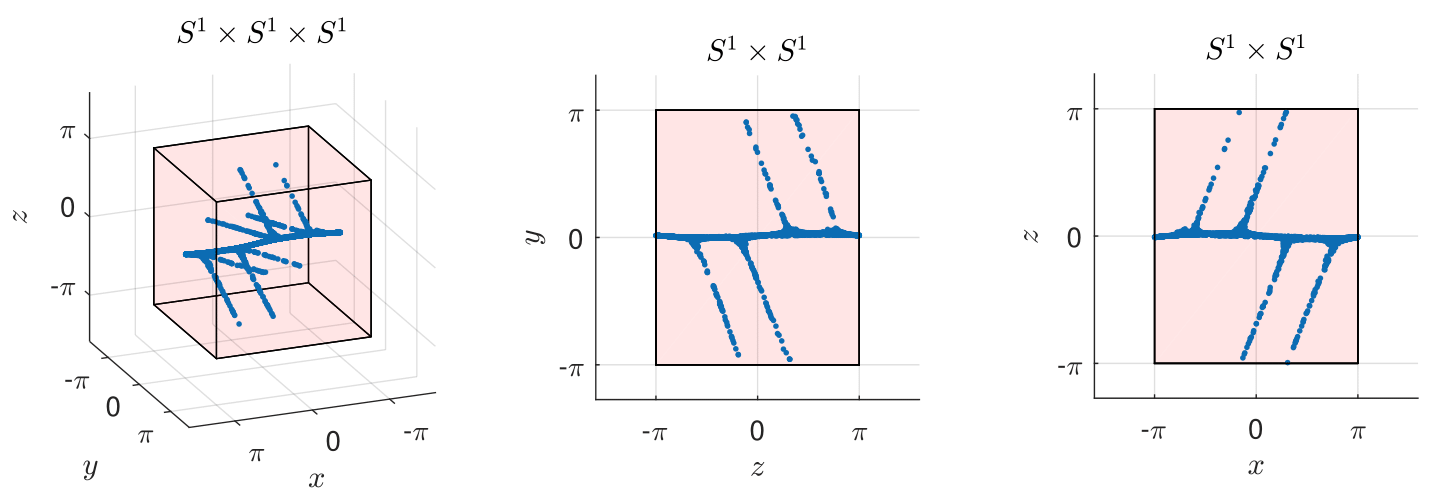}
  \caption{Scatter plot in the 3-torus (left) for $X$, along with two 2-d projections (center, left).
    The horizontal line on the $xy$ plane is a circle (the primary circle), and each one of the four $V$-shaped curves in $T^3$
is a hemisphere of a (secondary) circle.}
  \label{fig:three_circ_coords}
\end{figure}

As we can see from the scatter plot, these three coordinates provide a faithful realization of the data
in the three circle model proposed in \cite{carlsson2008local}.
Below in Figure \ref{fig:patches_three_circ_model} we show some of these image patches
in their $T^3$-coordinate to better illustrate what the actual circles are.

\begin{figure}[htb!]
  \centering
  \includegraphics[width=.6\textwidth]{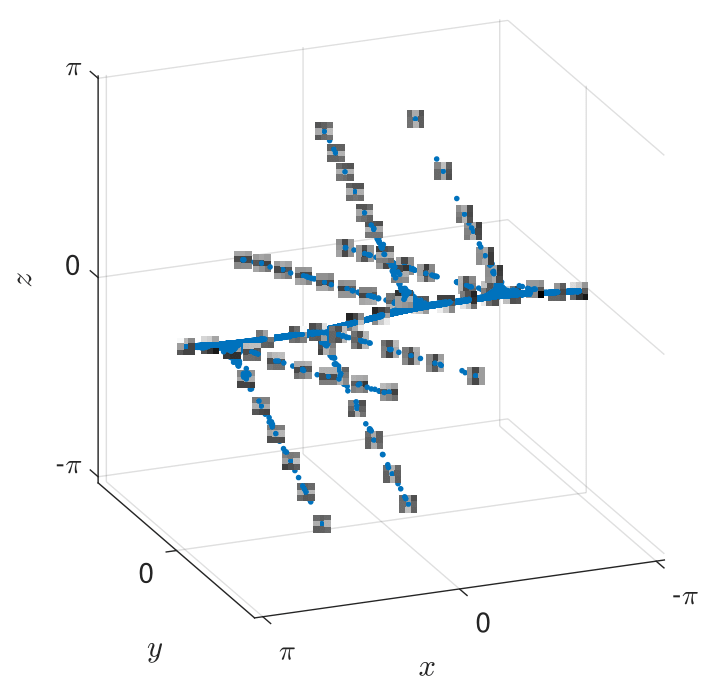}
  \caption{Image patches from $X$, plotted at their location in the 3-torus, according to the computed circular coordinates.}
  \label{fig:patches_three_circ_model}
\end{figure}

\section{Discussion}\label{sec:FinalRemarks}
We have presented in this paper an application of the theory of principal bundles
to the problem of finding topologically and geometrically meaningful coordinates for scientific data.
Specifically, we leverage the 1-dimensional persistent cohomology of the Rips filtration on
a subset of the data (the landmarks), in order to produce $S^1$-valued coordinates on the entire data set.
The coordinates are designed to capture 1-dimensional topological features of a continuous underlying space,
and the theory on which the coordinates are built, indicates that they classify $\Z$-principal bundles on the continuum.

The use of bundle theory allows for the circular coordinates to be sparse,
which is fundamental for analyzing geometric data of  realistic size.
We hope that these coordinates will be useful in problems such as
the analysis of recurrent dynamics in time series data (as in \cite{xu2018twisty}, \cite{tralie2018topological} or \cite{de2012topological}),
and nonlinear dimensionality reduction as indicated in the Experiments section \ref{sec:Experiments}.

An interesting direction from this work is the question of stability and Lipschitz continuity of sparse circular coordinates.
The main theoretical challenge is to determine how the edge and vertex weights on the Rips complex
can be used to stabilize the harmonic cocycle representative with respect to an appropriate notion of
(hopefully Hausdorff) noise on the landmark set. We hope to address this question in upcoming work.

\end{document}